\documentclass{amsproc}%
\usepackage{amsfonts}
\usepackage{amsmath}
\usepackage{amssymb}
\usepackage{graphicx}
\usepackage{hyperref}%
\providecommand{\U}[1]{\protect\rule{.1in}{.1in}}
\theoremstyle{plain}

\newtheorem{corollary}{Corollary}

\newtheorem{definition}{Definition}
\newtheorem{example}{Example}

\newtheorem{lemma}{Lemma}

\newtheorem{proposition}{Proposition}
\newtheorem{remark}{Remark}

\newtheorem{theorem}{Theorem}
\numberwithin{equation}{section}
\begin{document}
\title[{\normalsize On weakly classical 1-absorbing prime submodules}]{{\normalsize On weakly classical 1-absorbing prime submodules}}
\author{Zeynep Y\i lmaz U\c{c}ar}
\address{Department of Mathematics, Yildiz Technical University, Istanbul, Turkey;
Orcid: 0009-0009-8107-504X}
\email{zeynepyilmaz@hotmail.com.tr}
\author{Bayram Ali Ersoy}
\address{Department of Mathematics, Yildiz Technical University, Istanbul, Turkey;
Orcid: 0000-0002-8307-9644}
\email{ersoya@yildiz.edu.tr}
\author{\"{U}nsal Tekir}
\address{Department of Mathematics, Marmara University, Istanbul, Turkey;\\
Orcid: 0000-0003-0739-1449}
\email{utekir@marmara.edu.tr}
\author{Suat Ko\c{c}}
\address{Department of Mathematics, Istanbul Medeniyet University, Istanbul, Turkey;
Orcid: 0000-0003-1622-786X }
\email{suat.koc@medeniyet.edu.tr}
\author{Serkan Onar}
\address{Department of Mathematical Engineering, Yildiz Technical University, Istanbul,
Turkey; Orcid: 0000-0003-3084-7694}
\email{serkan10ar@gmail.com}

\begin{abstract}
In this paper, we study weakly classical 1-absorbing prime submodules of a
nonzero unital module $M\ $over a commutative ring $R$ having a nonzero
identity. A proper submodule $N$ of $M$ is said to be a weakly classical
1-absorbing prime submodule, if for each $m\in M$ and nonunits $a,b,c\in R,$
$0\neq abcm\in N$ implies that $abm\in N$ or $cm\in N$. We give various
examples and properties of weakly classical 1-absorbing prime submodules.
Also, we investiage the weakly classical 1-absorbing prime submodules of
tensor product$\ F\otimes M$ of a (faithfully) flat $R$-module $F$ and any
$R$-module $M.$ Also, we prove that if every proper submodule of an $R$-module
$M$ is weakly classical 1-absorbing prime, then $Jac(R)^{3}M=0.\ $In terms of
this result, we characterize modules over local rings in which every proper
submodule is weakly classical 1-absorbing prime.$\ $

\end{abstract}
\maketitle

\section{\bigskip Introduction}

In this article, all rings under consideration are assumed to be commutative
with nonzero identity and all modules are assumed to be nonzero unital. Let
$R$ always denote such a ring and $M$ denote such an $R$-module. A proper
submodule $N$ of $M\ $is called a \textit{prime submodule} if whenever $xm\in
N$ for some $x\in R$ and $m\in M,\ $then $x\in(N:_{R}M)$ or $m\in N,\ $where
$(N:_{R}M)$ is the annihilator $ann(M/N)\ $of $R$-module $M/N$, that is,
$(N:_{R}M)=\{x\in R:xM\subseteq N\}.\ $In this case, $(N:_{R}M)$ is a prime
ideal of $R.\ $However, the converse is not true in general. For instance,
consider $%
\mathbb{Z}
$-module $%
\mathbb{Z}
\oplus%
\mathbb{Q}
$ and note that $N=2%
\mathbb{Z}
\times%
\mathbb{Z}
$ is not a prime submodule since $2(0,\frac{1}{2})\in N$, $2\notin(N:_{%
\mathbb{Z}
}%
\mathbb{Z}
\oplus%
\mathbb{Q}
)$ and $(0,\frac{1}{2})\notin N$. But $(N:_{%
\mathbb{Z}
}%
\mathbb{Z}
\oplus%
\mathbb{Q}
)=0$ is a prime ideal of $%
\mathbb{Z}
$. Several authors have extended the notion of prime ideals to modules (See
\cite{X10,X13,X14}). In 2004, Behboodi and Koohy introduced the concept of
weakly prime submodules (See, \cite{X9}, \cite{X10} and \cite{XX}). According
to \cite{XX}, a proper submodule $N$ of $M$ is called a \textit{weakly prime}
\textit{submodule }if whenever $xym\in N$ for some $x,y\in R$ and $m\in M$,
then $xm\in N$ or $ym\in P$.\ By definition, it is clear that every prime
submodule is also weakly prime but the converse is not true in general. For
instance, the submodule $N=2%
\mathbb{Z}
\oplus0\ $of $%
\mathbb{Z}
$-module $%
\mathbb{Z}
\oplus%
\mathbb{Q}
$ is a weakly prime submodule which is not prime. It can be easily seen that a
proper submodule $N$ of $M\ $is weakly prime if and only if $(N:_{R}m)=\{x\in
R:xm\in N\}$ is a prime ideal of $R$ for every $m\notin N$. So far, the
concept of weakly prime submodules has drawn attention of many authors and has
been studied in many papers. For example, Azizi in his paper \cite[Lemma 3.2
and Theorem 3.3]{Azizi} studied the weakly prime submodules of tensor product
$F\otimes M$ for a flat (faithfully flat) $R$-module $F$ and any $R$-module
$M$ (See, \cite[Lemma 3.2 and Theorem 3.3]{Azizi}). Here, we should mention
that after some initial works, many authors studied the same notion under the
name "classical prime submodule", and in this paper we prefer to use classical
prime submodule instead of weakly prime submodule. In 2016, Mostafanasab et
al. gave a new generalization of classical prime submodules which is called
weakly classical prime submodule as follows: a proper submodule $N\ $of
$M\ $is called a weakly classical prime submodule if whenever $0\neq abm\in N$
for some $a,b\in R$ and $m\in M,\ $then $am\in N$ or $bm\in N$
\cite{Mostafanasab}$.\ $It is clear that zero submodule in any module is
trivially weakly classical prime. However, zero submodule need not be
classical prime in general. For instance, take the zero submodule
$N=(\overline{0})$ of $%
\mathbb{Z}
$-module $%
\mathbb{Z}
_{9}.\ $Then $N\ $is weakly classical prime which is not classical prime.

In a recent study, Yassine et al. gave a new generalization of prime ideal
which is called 1-absorbing prime ideal and they characterized local rings in
terms of 1-absorbing prime ideals. A proper ideal $Q$ of $R\ $is said to be a
1-absorbing prime ideal if whenever $xyz\in Q$ for some nonunits $x,y,z\in
R,\ $then $xy\in Q$ or $z\in Q$ \cite{Yassine}. Afterwards, Ko\c{c} et al.
defined the concept of weakly 1-absorbing prime ideals as follows: a proper
ideal $Q$ of $R\ $is said to be a weakly 1-absorbing prime if whenever $0\neq
xyz\in Q$ for some nonunits $x,y,z\in R,\ $then $xy\in Q$ or $z\in
Q\ $\cite{Eda}. Among many results in the paper \cite{Eda}, the authors
studied weakly 1-absorbing prime ideals in the ring of real valued continuous
functions $C(X)\ $on a topological space $X,\ $and also characterized rings in
which every proper ideal is weakly 1-absorbing prime. According to
\cite[Theorem 14]{Eda}, every proper ideal of a ring $R\ $is a weakly
1-absorbing prime if and only if $(R,M)$ is a local ring with $M^{3}=0\ $or
$R\ $is a direct product of two fields. Motivated by the studies of \cite{Eda}
and \cite{Mostafanasab}, in this paper we introduce and study weakly classical
1-absorbing prime submodules of modules over commutative rings. A proper
submodule $N$ of $M\ $is said to be a \textit{weakly} \textit{classical
1-absorbing prime submodule} if whenever $0\neq xyzm\in N$ for some nonunits
$x,y,z\in R$ and $m\in M,\ $then $xym\in N$ or $zm\in N.$ For the sake of
completeness, now we will give some notions and notations which will be
followed in the sequel. Let $R$ be a ring and $M\ $be an $R$-module. For every
element $x\in R$\ $(m\in M),\ $principal ideal (cyclic submodule)\ generated
by $x\in R\ $($m\in M)$ is denoted by $Rx$ ($Rm).\ $Also, for any subset $K$
of $M$ and any submodule $N$ of $M,\ $the residual of $N\ $by $K$\ will be
designated by $(N:_{R}K)=\{x\in R:xK\subseteq P\}.\ $In particular, if $N=(0)$
is the zero submodule and $K=\{m\}$ is a singleton for some $m\in M,\ $we
prefer to use $Ann_{R}(m)\ $instead of $((0):_{R}K).\ $Also, we prefer to use
$Ann_{R}(M)\ $to denote the annihilator of $M\ $instead of $((0):_{R}M)$
\cite{Sharp}. Recall from \cite{Sharp} that $M\ $is said to be a
\textit{faithful module} if $Ann_{R}(M)\ $is the zero ideal of $R.\ $Recall
from \cite{Barnard} that an $R$-module $M\ $is said to be a
\textit{multiplication module} if every submodule $N$ of $M\ $has the form
$IM$ for some ideal $I$ of $R.\ $In this case, $N=(N:_{R}M)M.\ $Ameri in his
paper \cite{Ameri} defined product of two submodules of multiplication modules
as follows: Let $M\ $be a multiplication module and $N=IM,\ K=JM$ be two
submodules of $M.\ $Then the product of $N\ $and $K\ $is defined by
$NK=IJM.\ $This product is independent of the represetantions of $N\ $and
$K\ $(See, \cite[Theorem 3.4]{Ameri}). In particular, we use $Nm$ to denote
the product of $N\ $and $Rm\ $for every submodule $N$ of $M\ $and $m\in M$.
Also for any submodule $N$ of $M\ $and any subset $J$ of $A$, the residual of
$N$ by $J$ is defined by $(N:_{M}J)=\{m\in M:Jm\subseteq N\}.\ $

Among other results in this paper, we investigate the relations between weakly
classical 1-absorbing prime submodules and weakly 1-absorbing prime ideals
(See, Theorem \ref{theo1}, Theorem \ref{text2} and Corollary \ref{cext1}). We
investigate the stability of weakly classical 1-absorbing prime submodules
under homomorphism, in localization and in factor modules (See, Theorem
\ref{thom}, Corollary \ref{cfac}, Corollary \ref{cfac2}, Theorem
\ref{tloc}).\ Recall from that \cite{Nico} that a proper submodule $N$ of
$M\ $is said to be a weakly 1-absorbing prime submodule if whenever $0\neq
abm\in N$ for some nonunits $a,b\in R$ and $m\in M,\ $then $ab\in(N:_{R}%
M)\ $or $m\in N.$ Recall from \cite{Ebra} that a proper submodule $N$ of
$M\ $is said to be a weakly semiprime if for each $a\in R$ and $m\in M,\ 0\neq
a^{2}m\in N$ implies that $am\in N.$ An $R$-module $M$ is said to be a reduced
module if whenever $a^{2}m=0$ for some $a\in R$ and $m\in M,\ $then
$am=0.\ $In Proposition \ref{pro1}, we investigate the relations between
weakly classical 1-absorbing prime submodules and other classical submodules
such as weakly 1-absorbing prime submodules, weakly classical prime submodules
and weakly semiprime submodules. Also, we prove that in any cyclic module,
weakly 1-absorbing prime submodules and weakly classical 1-absorbing prime
submodules are equal (See, Proposition \ref{pcyc}). We give very long two
lists whose all items are characterizing weakly classical 1-absorbing prime
submodules (See, Theorem \ref{tmain1} and Theorem \ref{tmain2}). Also, we
determine weakly classical 1-absorbing prime submodules of multiplication
modules (Theorem \ref{tmult}). In algebra, Prime avoidance lemma states that
if an ideal is contained in a finite union of prime ideals, then it must be
contained in one of those primes. If we replace prime ideals with any ideals,
then the claim is not true in general. For instance, consider $R=%
\mathbb{Z}
_{2}[x,y,z]/(x,y,z)^{2}$ and $I=\{\overline{0},\overline{x},\overline
{y},\overline{x+y}\}.\ $Let $P=\{\overline{0},\overline{x}\},\ Q=\{\overline
{0},\overline{y}\}$ and $J=\{\overline{0},\overline{x+y}\}.\ $Then $I\subseteq
P\cup Q\cup J$ and $I$ is not contained in one of the ideals $P,Q,J\ $of
$R.\ $One can naturally ask for which rings, a similar version of prime
avoidance lemma holds for any ideals: "When does $I\subseteq%
{\textstyle\bigcup\limits_{k=1}^{n}}
J_{k}\ $imply $I\subseteq J_{k}\ $for some $1\leq k\leq n\ $".\ In 1975,
Quartararo and Butts investigated the algebraic properties of such rings.
According to \cite{Quartararo} a ring $R\ $is said to be a $u$-ring if an
ideal is contained in a finite union of ideals, then it must be contained in
one of those ideals. Also, $R\ $is called a $um$-ring if $M\ $is an $R$-module
and $M=%
{\textstyle\bigcup\limits_{i=1}^{n}}
N_{i}\ $for some submodules $N_{i}\ $of $M,\ $then $M=N_{i}\ $for some $1\leq
i\leq n.$\ The authors in \cite{Quartararo} showed that Pr\"{u}fer domains and
B\.{e}zout rings are examples of u-rings. We give a characterization of weakly
classical 1-absorbing prime submodules of modules over $um$-rings (See,
Theorem \ref{tmain2}). Furthermore, we investigate the weakly classical
1-absorbing prime submodules of direct product of modules (See, Theorem
\ref{tcar1} and Theorem \ref{tcar2}). Recall from \cite{Larsen} that an
$R$-module $F$ is said to be a flat $R$-module if for each exact sequence
$K\rightarrow L\rightarrow M$ of $R$-modules, the sequence $F\otimes
K\rightarrow F\otimes L\rightarrow F\otimes M$ is also exact. Also, $F\ $is
said to be a faithfully flat, the sequence $K\rightarrow L\rightarrow M\ $is
exact if and only if $F\otimes K\rightarrow F\otimes L\rightarrow F\otimes M$
is exact. Azizi in \cite[Lemma 3.2]{Azizi} showed that if $M\ $is an
$R$-module, $N\ $is a submodule of $M\ $and $F\ $is a flat $A$-module, then
$(F\otimes N:_{F\otimes M}a)=F\otimes(N:_{M}a)$ for every $a\in R.$ In Theorem
\ref{theo16}, we characterize weakly classical 1-absorbing prime submodules of
tensor product$\ F\otimes M$ of a (faithfully) flat $R$-module $F$ and any
$R$-module $M.$ Recall from \cite{Sharp} that the intersection of all maximal
ideals of a ring $R\ $is denoted by $Jac(R)$ which is called the Jacobson
radical of $R.\ $Finally, we show that for any $R$-module $M,\ $if every
proper submodule of $M\ $is weakly classical 1-absorbing prime submodule, then
$Jac(R)^{3}M=0.\ $In terms of this result, for any module $M\ $over a local
ring $(R,\mathfrak{m}),\ $every proper submodule of $M\ $is weakly classical
1-absorbing prime if and only if $\mathfrak{m}^{3}M=0$ (See, Theorem
\ref{tfinal}).$\ $

\section{Characterization of weakly classical 1-absorbing prime submodules}

Throghout this section, $M\ $will denote a nonzero unital module over a
commutative ring $R\ $having a nonzero identity.

\begin{definition}
A proper submodule $N$ of an $R$-module $M$ is called a weakly classical
1-absorbing prime if whenever $0\neq abcm\in N\ $for some nonunits $a,b,c\in
R$ and $m\in M$, then $abm\in N$ or $cm\in N$.
\end{definition}

Recall from \cite{Zeynep} that a proper submodule $N$ of $M$ is said to be a
classical 1-absorbing prime submodule if $abcm\in N\ $for some nonunits
$a,b,c\in R$ and $m\in M$ implies that $abm\in N$ or $cm\in N.\ $By
definition, it is clear that every classical 1-absorbing submodule is also
weakly classical 1-absorbing submodule. However, the converse is not true in
general. See the following example.

\begin{example}
\textbf{(Weakly classical 1-absorbing prime }$\nRightarrow$\textbf{ Classical
1-absorbing prime)} Let $p,q,r$ be three distinct prime numbers and consider $%
\mathbb{Z}
$-module $M=%
\mathbb{Z}
_{pqr}.\ $Suppose that $N=(\overline{0})$. Then $N\ $is clearly a weakly
classical 1-absorbing. Since $p\cdot q\cdot r\cdot\overline{1}\in N$,\ $p\cdot
q\cdot\overline{1}\notin N$ and $r\cdot\overline{1}\notin N,\ $it follows that
$N\ $is not classical 1-absorbing prime submodule.
\end{example}

\begin{theorem}
\label{theo1} Let $M$ be an $R$-module and $N$ a proper submodule of $M$. If
$(N:_{R}m)$ is a weakly 1-absorbing prime ideal of $R$\ for each $m\in
M\diagdown N,\ $then $N$ is a weakly classical 1-absorbing prime submodule of
$M.$ The converse is true provided that $Ann_{R}(m)=0.$
\end{theorem}

\begin{proof}
Let $(N:_{R}m)$ be a weakly 1-absorbing prime ideal of $R$\ for each $m\in
M\diagdown N,\ $and choose $0\neq abcm\in N$ for some nonunits $a,b,c\in R$
and $m\in M.\ $Then note that $0\neq abc\in(N:_{R}m)$ and $m\in m\in
M\diagdown N.\ $By assumption, we have $ab\in(N:_{R}m)$ or $c\in(N:_{R}m)$,
that is, $abm\in N$ or $cm\in N.\ $For the converse, suppose that $N$ is a
weakly 1-absorbing submodule of $M,$ and choose $m\in M\diagdown N$ such that
$Ann_{R}(m)=0.\ $Let $0\neq abc\in(N:_{R}m)$ for some nonunits $a,b,c\in R$.
Then we get $0\neq abcm\in$ $N$. Since $N$ is a weakly 1-absorbing submodule
of $M$, we have either $abm\in N$ or $cm\in N$. Hence $ab\in(N:_{R}m)$ or
$c\in(N:_{R}m)$. Consequently, $(N:m)$ is a weakly 1-absorbing prime ideal of
$R$.
\end{proof}

In the converse of previous theorem, the condition " $Ann_{R}(m)=0$" is
necessary. On the other words, even if $N\ $is a weakly classical 1-absorbing
prime submodule of $M$ and $m\in M\diagdown N$ with $Ann_{R}(m)\neq0,\ $then
$(N:_{R}m)$ may not be weakly 1-absorbing prime ideal of $R.$ See, the
following example.

\begin{example}
Let $p,q$ be two distinct prime numbers and consider $%
\mathbb{Z}
$-module $M=%
\mathbb{Z}
_{p^{2}q^{2}}.\ $Let $N=(\overline{0})$ and $m=\overline{p}.\ $Then note that
$N\ $is a weakly classical 1-absorbing prime submodule and $m\in M\diagdown N$
with $Ann_{%
\mathbb{Z}
}(m)\neq0.\ $However, $(N:_{%
\mathbb{Z}
}M)=pq^{2}%
\mathbb{Z}
$ is not a weakly 1-absorbing prime ideal of $%
\mathbb{Z}
.$
\end{example}

Let $M$ be an $R$-module. Recall from \cite{Sharp} that the set of all torsion
elements is denoted by $T_{R}(M)=\{m\in M:Ann_{R}(m)=0\}.\ $It is worthy to
note that $T_{R}(M)$ is not always a submodule of $M.$ For instance, if we
consider $R=%
\mathbb{Z}
\times%
\mathbb{Z}
$-module $M=%
\mathbb{Q}
\times%
\mathbb{Q}
,\ $then clearly $T_{R}(M)=\left(  0\right)  \times%
\mathbb{Q}
\cup%
\mathbb{Q}
\times\left(  0\right)  $ is not a submodule of $M.\ $If $T_{R}(M)=M,\ M$ is
said to be a torsion module. If there exists $m\in M-T_{R}(M),\ $then we say
that $M$ is a non-torsion module.

\begin{theorem}
\label{text2}Let $M$ be a non-torsion module and $N\ $a proper submodule of
$M$ such that $T_{R}(M)\subseteq N.\ $Then $N\ $is a weakly classical
1-absorbing submodule of $M$ if and only if $(N:_{R}m)$ if a weakly
1-absorbing prime ideal of $R\ $for each $m\in M\diagdown N.$
\end{theorem}

\begin{proof}
Assume that $T_{R}(M)\subseteq N.\ $Then note that $M-N\subseteq M-T_{R}(M)$
which means that $Ann_{R}(m)=0$ for each $m\in M-N.\ $The rest follows from
Theorem \ref{theo1}.
\end{proof}

\begin{corollary}
\label{cext1}Let $R\ $be a ring and $I\ $a proper ideal of $R.\ $Then $I\ $is
a weakly classical 1-absorbing prime submodule of $R$-module $R\ $if and only
if $I\ $is a weakly 1-absorbing prime ideal of $R.$
\end{corollary}

\begin{proof}
Let $I$ be a weakly classical 1-absorbing prime submodule of $R$-module
$R.\ $Then by Theorem \ref{theo1}, we have $(I:_{R}1)=I$ is a weakly
1-absorbing prime ideal of $R.\ $Now, let $I\ $be a weakly 1-absorbing prime
ideal and choose nonunits $a,b,c\in R$ and $m\in R$ such that $0\neq
abcm=ab(cm)\in I.\ $Since $I\ $is a weakly 1-absorbing prime ideal, we have
$ab\in I$ or $cm\in I.\ $Which implies that $abm\in I\ $or $cm\in I.$
\end{proof}

\begin{theorem}
\label{thom} Let $M,M^{\prime}$ be two $R$-modules and $f:M\rightarrow
M^{\prime}$ be an $R$-homomorphism.

(1) Suppose that $f$ is a monomorphism and\ $N^{\prime}\ $is a weakly
classical 1-absorbing prime submodule of $M^{\prime}\ $with $f^{-1}\left(
N^{\prime}\right)  \neq M$, then $f^{-1}\left(  N^{\prime}\right)  $ is a
weakly classical 1-absorbing submodule of $M$.

(2) Suppose that $f$ is an epimorphism and\ $N$ is a weakly classical
1-absorbing prime submodule of \ $M$ containing $Ker(f)$, then $f\left(
N\right)  $ is a weakly classical 1-absorbing prime submodule of $M^{\prime}$.
\end{theorem}

\begin{proof}
$(1):\ $Suppose that $N^{\prime}$ is a weakly classical 1-absorbing prime
submodule of\ $M^{\prime}\ $with $f^{-1}\left(  N^{\prime}\right)  \neq M$.
Let $0\neq abcm\in f^{-1}\left(  N^{\prime}\right)  $ for some nonunits
$a,b,c\in R$ and $m\in M$. Since $f$ is monomorphism, $0\neq
f(abcm)=abcf(m)\in N^{\prime}$. Since $N^{\prime}$ is a weakly classical
1-absorbing prime submodule, we have $abf(m)\in N^{\prime}$ or $cf(m)\in
N^{\prime}$. Which implies that $f(abm)\in N^{\prime}$ or $f(cm)\in N^{\prime
}$. Therefore, $abm\in f^{-1}\left(  N^{\prime}\right)  $ or $cm\in
f^{-1}\left(  N^{\prime}\right)  $. Consequently,$\ f^{-1}\left(  N^{\prime
}\right)  $ is a weakly classical 1-absorbing prime submodule of $M$.

$(2):\ $Assume that $N$ is a weakly classical 1-absorbing prime submodule
of\ $M$. Let $a,b,c\in R$ be nonunits and $m^{\prime}\in M^{\prime}\ $be such
that $0\neq abcm^{\prime}\in f\left(  N\right)  $. By assumption, there exists
$m\in M$ such that $m^{\prime}=f(m)$ and so $0\neq abcm^{\prime}=f(abcm)\in
f(N)$. Since $Ker(f)\subseteq N$, we have $0\neq abcm\in N$. As $N$ is a
weakly classical 1-absorbing prime submodule of $M,$ we have either $abm\in N$
or $cm\in N$. Hence, $f(abm)=abm^{\prime}\in f(N)$ or $f(cm)=cm^{\prime}\in
f(N)$. Consequently, $f(N)$ is a weakly classical 1-absorbing prime submodule
of $M^{\prime}.$
\end{proof}

As an immediate consequence of Theorem \ref{thom} $\left(  2\right)  ,$ we
have the following explicit result.

\begin{corollary}
\label{cfac}Let $M$ be an $R$-module and $L\subseteq N$ be two submodules of
$M$. If $N$ is a weakly classical 1-absorbing prime submodule of $M$, then
$N/L$ is a weakly classical 1-absorbing prime submodule of $M/L$.
\end{corollary}

The converse of previous corollary need not be true. See, the following example.

\begin{example}
Let $p$ be a prime number and consider $%
\mathbb{Z}
$-module $M=%
\mathbb{Z}
_{p^{3}}^{3}$. Let $L=%
\mathbb{Z}
_{p^{3}}\times(\overline{0})\times(\overline{0})=N.\ $Then $N/L$ is a zero
submodule which is clearly weakly classical 1-absorbing prime. Since
$p^{3}(\overline{1},\overline{1},\overline{1})\in N\ $and$\ p^{2}(\overline
{1},\overline{1},\overline{1})\notin N,\ $it follows that $N\ $is not a weakly
classical 1-absorbing prime submodule of $M.$
\end{example}

\begin{corollary}
\label{cfac2}Let $K$ and $N$ be two submodules of $M$ with $K\subseteq
N\subseteq M$. If $K$ is a weakly classical 1-absorbing prime submodule of $M$
and $N/K$ is a weakly classical 1-absorbing prime submodule of $M/K$, then $N$
is a weakly classical 1-absorbing prime submodule of $M$.
\end{corollary}

\begin{proof}
Assume that $K$ is a weakly classical 1-absorbing prime submodule of $M$ and
$N/K$ is a weakly classical 1-absorbing prime submodule of $M/K$. Let
$a,b,c\in R$ be nonunits and $m\in M$ with $0\neq abcm\in N$. Assume that
$abcm\in K$. Since $K$ is a weakly classical 1-absorbing prime submodule, we
get either $abm\in K\subseteq N$ or $cm\in K\subseteq N$. Now, assume that
$abcm\notin K$. Then we obtain $0\neq abc(m+K)\in N/K$. Since $N/K$ is a
weakly classical 1-absorbing prime submodule, we have $ab(m+K)\in N/K$ or
$c(m+K)\in N/K$ which implies that $abm\in N$ or $cm\in N$, as needed.
\end{proof}

\begin{theorem}
\label{tloc}Let $M$ \ be an $R$-module, $N$ be a submodule of $M\ $and $S$ be
a multiplicative subset of $R$. If $N$ is a weakly classical 1-absorbing prime
submodule of $M$ such that $\left(  N:_{R}M\right)  \cap S=\varnothing$, then
$S^{-1}N$ is a weakly classical 1-absorbing prime submodule of $S^{-1}M$.
\end{theorem}

\begin{proof}
Let $N$ be a weakly classical 1-absorbing prime submodule of $M$ and $\left(
N:_{R}M\right)  \cap S=\varnothing$. Suppose that $0\neq\frac{a}{s}\frac{b}%
{t}\frac{c}{u}\frac{m}{v}\in S^{-1}N$ for some nonunits $\frac{a}{s},\frac
{b}{t},\frac{c}{u}\in S^{-1}R\ $and $\frac{m}{v}\in S^{-1}M$. Then there
exists $w\in S$ such that $0\neq wabcm\in N$. Since $N$ is a weakly classical
1-absorbing prime, then we have $ab(wm)\in N$ or $c(wm)\in N$. Thus $\frac
{a}{s}\frac{b}{t}\frac{m}{v}=\frac{wabm}{wstv}\in S^{-1}N$ or $\frac{c}%
{u}\frac{m}{v}=\frac{wcm}{wuv}\in S^{-1}N$. Therefore, $S^{-1}N$ is a
classical 1-absorbing prime submodule of $S^{-1}M$.
\end{proof}

\begin{proposition}
\label{pro1} Let $N$ be a proper submodule of $M.$

(1)\ If $N$ is a weakly 1-absorbing prime submodule of $M$, then $N$ is a
weakly classical 1-absorbing prime submodule of $M$.

(2) If $N$ is a weakly classical prime submodule of $M,\ $then $N\ $is a
weakly classical 1-absorbing prime submodule and weakly semiprime submodule of
$M.\ $

(3) Suppose that $M$ is a reduced module. Then $N\ $is a weakly classical
prime submodule of $M$ if and only if $N\ $is a weakly semiprime submodule and
weakly classical 1-absorbing prime submodule of $M.$
\end{proposition}

\begin{proof}
$(1),(2):\ $Clear.

$(3)\ \left(  \Rightarrow\right)  :$ follows from (2). $\left(  \Leftarrow
\right)  $: suppose that $N$ is a weakly classical 1-absorbing prime submodule
and weakly semiprime submodule of a reduced module $M$. Choose $a,b\in R$ and
$m\in M$ such that $0\neq abm\in N$.\ Without loss of generality, we may
assume that $a,b$ are nonunits$.$ Since $M\ $is reduced, we have $0\neq
a^{2}bm\in N.\ $As $N\ $is a weakly classical 1-absorbing prime submodule, we
get either $a^{2}m\in N$ or $bm\in N.\ $If $a^{2}m\in N,\ $then $am\in
N\ $because $N\ $is weakly semiprime. Thus, $N\ $is weakly classical prime
submodule of $M.$
\end{proof}

\begin{proposition}
\label{pcyc}Let $M$ be a cyclic $R$-module. Then a proper submodule $N$ of $M$
is a weakly 1-absorbing prime submodule if and only if it is weakly classical
1-absorbing prime submodule of $M$.
\end{proposition}

\begin{proof}
$\left(  \Rightarrow\right)  \ $Follows from Proposition \ref{pro1}
(1).\ $\left(  \Leftarrow\right)  $ Let $M=Rm$ for some $m\in M$ and $N$ be a
weakly classical 1-absorbing prime submodule of $M$. Suppose that $0\neq
abx\in N$ for some nonunits $a,b\in R$ and $x\in M$. If $Rx=Rm=M$, then we
conclude that $ab\in(N:_{R}M).\ $So assume that $Rx\neq M.\ $Then there exists
nonunit $c\in R$ such that $x=cm$. Therefore $0\neq abx=abcm\in$ $N$.
Since$\ N$ is a weakly classical 1-absorbing prime submodule of $M,$ we have
either $abm\in N$ or $cm\in N$. Hence $ab\in(N:_{R}M)$ or $x=cm\in
N$.\ Consequently, $N$ is a classical 1-absorbing prime submodule of $M$.
\end{proof}

\begin{definition}
Let $N$ be a proper submodule of $M$ and $a,b,c$ be nonunits in $R$, $m\in M$.
If $N$ is a weakly classical 1-absorbing prime submodule and $abcm=0$,
$abm\notin N$, $cm\notin N$, then $\left(  a,b,c,m\right)  $ is called a
classical 1-quadruple-zero of $N$.
\end{definition}

\begin{theorem}
\label{theo5} Let $N$ be a weakly classical 1-absorbing prime submodule of
$M,$ and suppose that $abcK\subseteq N$ for some nonunits $a,b,c\in R$ and
some submodule $K$ of $M$. If $\left(  a,b,c,k\right)  $ is not a classical
1-quadruple-zero of $N$ for every $k\in K$, then $abK\subseteq N$ or
$cK\subseteq N$.
\end{theorem}

\begin{proof}
Suppose that $\left(  a,b,c,k\right)  $ is not a classical 1-quadruple-zero of
$N$ for every $k\in K$ and $abcK\subseteq N$. Assume on the contrary that
$abK\nsubseteq N$ and $cK\nsubseteq N$. Then there are $k_{1},k_{2}\in K$ such
that $abk_{1}\notin N$ and $ck_{2}\notin N$. If $abck_{1}\neq0$, then we have
$ck_{1}\in N$, because $abk_{1}\notin N$ and $N\ $is a weakly classical
1-absorbing prime. So assume that $abck_{1}=0.\ $Since $\left(  a,b,c,k_{1}%
\right)  $ is not a classical 1-quadruple-zero of $N$, we conclude that
$ck_{1}\in N.\ $In both cases, we have $ck_{1}\in N.\ $Likewise, $abk_{2}\in
N.$ On the other hand, note that $abc(k_{1}+k_{2})\in N.\ $A similar argument
above shows that $ab(k_{1}+k_{2})\in N$ or$\ c(k_{1}+k_{2})\in N.\ $As
$abk_{2}\in N$ and $ck_{1}\in N,\ $we have $abk_{1}\in N$ or $ck_{2}\in
N\ $which both of them are contradictions. Hence, we obtain $abK\subseteq N$
or $cK\subseteq N$.
\end{proof}

\begin{definition}
\bigskip Let $N$ be a weakly classical 1-absorbing prime submodule of $M,$ and
suppose that $IJLK\subseteq N$ for some proper ideals $I,J,L$ of $R$ and some
submodule $K$ of $M$. We say that $N$ is a free classical 1-quadruple-zero
with respect to $IJLK$ if $\left(  a,b,c,k\right)  $ is not a classical
1-quadruple-zero of $N$ for every elements $a\in I$, $b\in J$, $c\in L$ and
$k\in K$.
\end{definition}

\begin{remark}
\bigskip Let $N$ be a weakly classical 1-absorbing prime submodule of $M$ and
suppose that $IJLK\subseteq N$ for some proper ideals $I,J,L$ of $R$ and some
submodule $K$ of $M$ such that $N$ is a free classical 1-quadruple-zero with
respect to $IJLK$. Hence, if $a\in I$, $b\in J$, $c\in L$ and $k\in K,$\ then
$abk\in N$ or $ck\in N$.
\end{remark}

\begin{corollary}
Let $N$ be a weakly classical 1-absorbing prime submodule of $M,$ and suppose
that $IJLK\subseteq N$ for some proper ideals $I,J,L$ of $R$ and some
submodule $K$ of $M$. If $N$ is a free classical 1-quadruple-zero with respect
to $IJLK$, then $IJK\subseteq N$ or $LK\subseteq N$.
\end{corollary}

\begin{proof}
Suppose that $N$ is a free classical 1-quadruple-zero with respect to $IJLK$.
Assume that $IJK\nsubseteq N$ and $LK\nsubseteq N$. Then there are $a\in I$,
$b\in J$, $c\in L$ with $abK\nsubseteq N$ and $cK\nsubseteq N$. Since
$abcK\subseteq N$ and $N$ is a free classical 1-quadruple-zero with respect to
$IJLK$, then Theorem \ref{theo5} implies that $abK\subseteq N$ and
$cK\subseteq N$, which is a contradiction. Consequently, we have $IJK\subseteq
N$ or $LK\subseteq N$.
\end{proof}

Let $M$ be an $R$-module and $N$ be a submodule $M$. For every $a\in R$,
$\left\{  m\in M:am\in N\right\}  $ is denoted by $\left(  N:_{M}a\right)  $.
It is easy to see that $\left(  N:_{M}a\right)  $ is a submodule of $M$
containing $N$. In the next theorem, we give a long list whose each item
characterizes weakly classical 1-absorbing prime submodules.

\begin{theorem}
\label{tmain1}Let $M$ be an $R$-module and $N$ be a proper submodule $M$. The
following conditions are equivalent.

(1) $N$ is a weakly classical 1-absorbing prime submodule;

(2) For every nonunits $a,b,c\in R$, $\left(  N:_{M}abc\right)  =\left(
0:_{M}abc\right)  \cup\left(  N:_{M}ab\right)  \cup\left(  N:_{M}c\right)  $;

(3)\ For every nonunits $a,b,c\in R$ and $m\in M$ with $abm\notin N$; $\left(
N:_{R}abm\right)  =\left(  0:_{R}abm\right)  \cup\left(  N:_{R}m\right)  $;

(4)\ For every nonunits $a,b,c\in R$ and $m\in M$ with $abm\notin N$; $\left(
N:_{R}abm\right)  =\left(  0:_{R}abm\right)  $ or $\left(  N:_{R}abm\right)
=\left(  N:_{R}m\right)  $;

(5) For every nonunits $a,b\in R$, $m\in M$ and every proper ideal $I$ of $R$
with $0\neq abIm\subseteq N,$ either $abm\in N$ or $Im\subseteq N$;

(6) For every nonunits $a,b\in R$, $m\in M$ and every proper ideal $I$ of $R$
with $0\neq abIm\subseteq N,$ either $aIm\subseteq N$ or $bm\in N$;

(7) For every proper ideal $I$ of $R,$ nonunit $a\in R$ and $m\in M$ with
$a$I$m\nsubseteq N$, $\left(  N:_{R}aIm\right)  =\left(  0:_{R}aIm\right)
\ $or $\left(  N:_{R}aIm\right)  =\left(  N:_{R}m\right)  $;

(8) For every proper ideals $I,J,K$ of $R$ and $m\in M$ with $0\neq
IJKm\subseteq N$, either $IJm\subseteq N$ or $Km\subseteq N$.
\end{theorem}

\begin{proof}
$(1)\Rightarrow(2):$ Suppose that $N$ is a weakly classical 1-absorbing prime
submodule of $M$. Let $m\in$ $\left(  N:_{M}abc\right)  $. Then $abcm\in N$.
If $abcm=0$, then $m\in\left(  0:_{M}abc\right)  $. Assume that $abcm\neq
0.$Since $N$ is a weakly classical 1-absorbing prime submodule, we have
$abm\in N$ or $cm\in N$. Hence $m\in\left(  N:_{M}ab\right)  $ or $m\in\left(
N:_{M}c\right)  $ and so $m\in\left(  0:_{M}abc\right)  \cup\left(
N:_{M}ab\right)  \cup\left(  N:_{M}c\right)  $. Consequently, $\left(
N:_{M}abc\right)  =\left(  0:_{M}abc\right)  \cup\left(  N:_{M}ab\right)
\cup\left(  N:_{M}c\right)  $.

$(2)\Rightarrow(3):\ $Let $abm\notin N$ for some nonunits $a,b\in R$ and $m\in
M$. Assume that $c\in\left(  N:_{R}abm\right)  $. Then $abcm\in N$, and so
$m\in\left(  N:_{M}abc\right)  $.Thus by part (2), $m\in\left(  0:_{M}%
abc\right)  \cup\left(  N:_{M}ab\right)  \cup\left(  N:_{M}c\right)  $. Since
$abm\notin N$, then $m\notin\left(  N:_{M}ab\right)  $ which implies that
$m\in\left(  0:_{M}abc\right)  $ or $m\in\left(  N:_{M}c\right)  $. Therefore
$c\in\left(  0:_{R}abm\right)  $ or $c\in\left(  N:_{R}m\right)  $.
Consequently, $c\in\left(  0:_{R}abm\right)  \cup\left(  N:_{R}m\right)  $.
Thus $\left(  N:_{R}abm\right)  =\left(  0:_{R}abm\right)  \cup\left(
N:_{R}m\right)  $.

$(3)\Longrightarrow(4):\ $By the fact that if an ideal (a subgroup) is the
union of two ideals (two subgroups), then it is equal to one of them.

$(4)\Longrightarrow(5):\ $Assume that $0\neq abIm\subseteq N\ $for some
nonunits $a,b\in R$, $m\in M$ and a proper ideal $I$ of $R$. Hence
$I\subseteq(N:_{R}abm)$ and $I\nsubseteq(0:_{R}abm)$. If $abm\in N$, then we
are done. So, assume that $abm\notin N$. Therefore by part (4), we have that
$I\subseteq\left(  N:_{R}m\right)  $ which implies that $Im\subseteq N$.

$(5)\Longrightarrow(6):\ $Let $0\neq abIm\subseteq N\ $for some proper ideal
$I$ of $R,$ nonunits $a,b\in R\ $and $m\in M$. Assume that $aIm\nsubseteq
N.\ $Then there exists $x\in I$ such that $axm\notin N.\ $Then note that
$ax(Rb)m\subseteq N.\ $If $0\neq ax(Rb)m,$ then by part (5) we have
$(Rb)m\subseteq N$ which completes the proof. So asusme that $ax(Rb)m=0.\ $%
Since $0\neq abIm,\ $there exists $y\in I$ such that $ay(Rb)m\neq0.\ $This
implies that $0\neq a(x+y)(Rb)m\subseteq N$.\ As $0\neq ay(Rb)m\subseteq
N,\ $again by part (5), we have $aym\in N$ or $(Rb)m\subseteq N.\ $\ Let
$aym\in N.\ $Then we have $a(x+y)m\notin N.\ $As $0\neq a(x+y)(Rb)m\subseteq
N,\ $by part (5),\ we get $(Rb)m\subseteq N$ which implies that $bm\in N.$ In
the other case, we have $bm\in N$ which completes the proof.

$(6)\Rightarrow(7):\ $Let $a$I$m\nsubseteq N$ for some nonunit $a\in R,\ $some
proper ideal $I$ of $R$ and $m\in M.\ $Choose $b\in(N:_{R}a$I$m).\ $Then we
have $ab$I$m\subseteq N.\ $If $ab$I$m=0,\ $then \ we get $b\in(0:_{R}a$I$m).$
Now, assume that $ab$I$m\neq0.\ $Then by part (6), we get $bm\in N$ which
implies that $b\in(N:_{R}m).\ $Thus, we conclude that $\left(  N:_{R}%
aIm\right)  =\left(  0:_{R}aIm\right)  \cup(N:_{R}m).\ $This implies that
$\left(  N:_{R}aIm\right)  =\left(  0:_{R}aIm\right)  \ $or $\left(
N:_{R}aIm\right)  =\left(  N:_{R}m\right)  $.

$(7)\Longrightarrow(8):\ $Let $0\neq IJKm\subseteq N$ for some proper ideals
$I,J,K\ $of $R$ and $m\in M.\ $Assume that $IJm\nsubseteq N.\ $Then there
exists $a\in J\ $such that $a$I$m\nsubseteq N.\ $Since $aIKm\subseteq N,\ $we
have $K\subseteq\left(  N:_{R}aIm\right)  =\left(  0:_{R}aIm\right)  $ or
$K\subseteq\left(  N:_{R}aIm\right)  =\left(  N:_{R}m\right)  $ by part
(7).\ This gives $aIKm=0$ or $Km\subseteq N.\ $In the second case, we are
done. So we may assume that $aIKm=0.\ $Since $0\neq IJKm,\ $there exists $b\in
J$ such that $IbKm\neq0.\ $Since $K\subseteq\left(  N:_{R}bIm\right)
,\ $again by part (7), we have $bIm\subseteq N$ or $Km\subseteq N.\ $So assume
that $bIm\subseteq N$.\ Then $(a+b)Im\nsubseteq N$ and $0\neq(a+b)IKm\subseteq
N.\ $Since $K\subseteq(N:_{R}(a+b)Im),\ $by part (7), we conclude that
$Km\subseteq N.$

$(8)\Rightarrow(1):\ $Clear.
\end{proof}

\begin{theorem}
\label{theo7} Let $N$ be a weakly classical 1-absorbing prime submodule of $M$
and suppose that $\left(  a,b,c,m\right)  $ is a classical 1-quadruple-zero of
$N$ for some nonunits $a,b,c\in R$ and $m\in M$. Then

(1) $abcN=0=ab(N:_{R}M)m.$

(2) If $a,b\notin(N:_{R}cm),\ $then $ac(N:_{R}M)m=bc(N:_{R}M)m=a(N:_{R}%
M)^{2}m=b(N:_{R}M)^{2}m=c(N:_{R}M)^{2}m=0.\ $In particular, $(N:_{R}%
M)^{3}m=0.\ $
\end{theorem}

\begin{proof}
$(1):\ $Let $abcN\neq0.\ $Then there exists $n\in N\ $with $abcn\neq0.\ $This
gives $0\neq abcn=abc(m+n)\in N.\ $Since $N$ is weakly classical 1-absorbing
prime, we have $ab(m+n)\in N$ or $c(m+n)\in N.\ $Which implies that $abm\in N$
or $cm\in N,\ $which is a contradiction. Thus, $abcN=0.$\ Now, we will show
that $ab(N:_{R}M)m=0.\ $Choose $x\in(N:_{R}M)$ and assume that $abxm\neq
0.\ $Then we have $0\neq abxm=ab(c+x)m\in N.\ $If $c+x\ $is unit, we conclude
that $abm\in N\ $which is a contradiction. Thus $c+x$ is nonunit. As $N\ $is a
weakly classical 1-absorbing prime, we have $abm\in N\ $or $(c+x)m\in
N.\ $Then we obtain $abm\in N\ $or $cm\in N,\ $again a contradiction. Thus,
$ab(N:_{R}M)m=0.\ $

$(2):\ $Let $a,b\notin(N:_{R}cm),\ $that is, $acm\notin N$ and $bcm\notin
N.\ $Choose $x\in(N:_{R}M)$ and $acxm\neq0.\ $Then we have $0\neq a(b+x)cm\in
N.\ $Note that $b+x$ must be nonunit. As $N\ $is weakly 1-absorbing prime,
$a(b+x)m\in N$ or $cm\in N\ $which yields that $abm\in N$ or $cm\in N.$ This
is a contradiction. Hence, we have $ac(N:_{R}M)m=0.\ $Likewise, we have
$bc(N:_{R}M)m=0.\ $Now, we will show that $a(N:_{R}M)^{2}m=0.\ $Suppose that
$a(N:_{R}M)^{2}m\neq0.\ $Then there exist $x,y\in(N:_{R}M)$ such that
$axym\neq0.\ $Then by $ab(N:_{R}M)m=ac(N:_{R}M)m=0,\ $we have $0\neq
axym=a(b+x)(c+y)m\in N.\ $Since $abm\notin N$ and$\ cm\notin N,$ $(b+x)$ and
$(c+y)$ are nonunits. As $N\ $is weakly classical 1-absorbing prime, we have
$a(b+x)m\in N$ or $(c+y)m\in N.\ $Which implies that $abm\in N\ $or $cm\in N$.
This is a contradiction. Thus, we have $a(N:_{R}M)^{2}m=0.\ $Similarly one can
prove that $b(N:_{R}M)^{2}m=c(N:_{R}M)^{2}m=0.$\ Now, we will show that
$(N:_{R}M)^{3}m=0.\ $Suppose to the contrary. Then there exist $x,y,z\in
(N:_{R}M)$ such that $xyzm\neq0.\ $Then we have $0\neq(a+x)(b+y)(c+z)m=xyzm\in
N.\ $Also, $(a+x),(b+y),(c+z)$ are nonunits. Then we conclude that
$(a+x)(b+y)m\in N$ or $(c+z)m\in N.\ $This implies that $abm\in N$ or $cm\in
N\ $which is a contradiction. Hence, $(N:_{R}M)^{3}m=0.$
\end{proof}

A submodule $N$ of an $R$-module $M$ is called a \textit{nilpotent submodule}
if $(N:_{R}M)^{k}N=0$ for some positive integer $k$ \cite{Ali}. In particular,
an element $m\in M$ is nilpotent in $M$ if $\ Rm$ is a nilpotent submodule of
$M$.

\begin{theorem}
\label{theo8} If $N$ is a weakly classical 1-absorbing prime submodule of $M$
that is not classical 1-absorbing prime submodule, then there exists classical
1-quadruple-zero $\left(  a,b,c,m\right)  $ of $N.\ $If $a,b\notin
(N:_{R}cm),\ $then $(N:_{R}M)^{3}N=0$ and thus $N$ is nilpotent. Furthermore,
if $M$ is a multiplication module, then $N^{4}=0.\ $
\end{theorem}

\begin{proof}
Suppose that $N$ is a weakly classical 1-absorbing prime submodule of $M$ that
is not a classical 1-absorbing prime submodule. Then there exists a classical
1-quadruple-zero $\left(  a,b,c,m\right)  $ of $N$ for some nonunits $a,b,c\in
R$ and $m\in M$. Suppose that $a,b\notin(N:_{R}cm).\ $Now, we will show that
$(N:_{R}M)^{3}N=0.\ $Suppose that $(N:_{R}M)^{3}N\neq0$. Then there are
$x,y,z\in(N:_{R}M)$ and $n\in N$ such that $xyzn\neq0$. By Theorem
\ref{theo7}, we have $0\neq(a+x)(b+y)(c+z)(m+n)=xyzn\in N$. Since $N$ is a
weakly classical 1-absorbing prime submodule, we have $(a+x)(b+y)(m+n)\in N$
or $(c+z)(m+n)\in N$. Then we have $abm\in N$ or $cm\in N$ which is a
contradiction. Thus, we obtain $(N:_{R}M)^{3}N=0$ and $N$ is nilpotent. Now,
suppose that $M\ $is a multiplication module. Then it is clear that
$N^{4}=(N:M)^{3}N=0.$
\end{proof}

\begin{proposition}
Let $M$ be a multiplication module and $N$ a proper submodule of $M$. The
following conditions are equivalent.

(1)$\ N$ is a weakly classical 1-absorbing prime submodule of $M$.

(2)\ If $0\neq$ $KLPm\subseteq N$ for some proper submodules $K,L,P\ $of $M$
and $m\in M$, then either $KLm\subseteq N$ or $Pm\subseteq N$ .
\end{proposition}

\begin{proof}
$(1)\Rightarrow(2):\ $Let $0\neq KLPm\subseteq N$ for some proper submodules
$K,L,P$\ of $M$ and $m\in M$. Since $M$ is a multiplication module, there are
ideals $I,J,Q$ of $R$ such that $K=IM,L=JM$ and $P=QM$. Since $0\neq
KLPm\subseteq N$ we have $0\neq IJQm\subseteq N,$ and thus we have
$IJm\subseteq N$ or $Qm\subseteq N$. Hence $KLm\subseteq N$ or $Pm\subseteq N$.

$(2)\Rightarrow(1):\ $Suppose that $0\neq IJQm\subseteq N$ for some proper
ideals $I,J,Q\ $of $R$ and some $m\in M$. It is sufficient to set
$K:=IM$,$\ L:=JM$ and $P:=QM$ in part (2). Hence $0\neq KLPm\subseteq N$
implies that $KLm\subseteq N$ or $Pm\subseteq N$. This implies that
$IJm\subseteq N$ or $Qm\subseteq N.\ $Consequently, $N$ is a weakly classical
1-absorbing prime submodule of $M$ by Theorem \ref{tmain1}.
\end{proof}

Now, we are ready to give a new characterization of weakly classical
1-absorbing prime submodules of modules over um-rings.

\begin{theorem}
\label{tmain2}Let $R$ be a um-ring and $M$ be an $R$-module. For any proper
submodule $N$ of $M,\ $the followings are equivalent.

(1)\ $N\ $is a weakly classical 1-absorbing prime submodule.

(2) For every nonunits $a,b,c\in R,\ (N:_{M}abc)=(N:_{M}ab)$ or $(N:_{M}%
abc)=(0:_{M}abc)$ or $(N:_{M}abc)=(N:_{M}c).\ $

(3)\ For every nonunits $a,b,c\in R$ and every submodule $L$ of $M,\ 0\neq
abcL\subseteq N$ implies that $abL\subseteq N$ or $cL\subseteq N.$

(4)\ For every nonunits $a,b\in R\ $and every submodule $L$ of $M$ with
$abL\nsubseteq N,\ (N:_{R}abL)=(0:_{R}abL)$ or $(N:_{R}abL)=(N:_{R}L).$

(5) For every nonunits $a,b\in R,\ $every proper ideal $J$ of $R\ $and every
submodule $L$ of $M$ with $0\neq abJL\subseteq N$ implies that $abL\subseteq
N$ or $JL\subseteq N.$

(6)\ For every nonunits $a\in R,\ $every proper ideals $I,J\ $of $R$ and every
submodule $L\ $of $M\ $with $0\neq aIJL\subseteq N$ implies that $aIL\subseteq
N$ or $JL\subseteq N.\ $

(7) For every proper ideals $I,J,K\ $of $R\ $and every submodule $L\ $of $M$
with $IJL\nsubseteq N,\ (N:_{R}IJL)=(0:_{R}IJL)$ or $(N:_{R}IJL)=(N:_{R}L).$

(8)\ For every proper ideals $I,J,K\ $of $R$ and every submodule $L$ of $M$,
$0\neq IJKL\subseteq N$ implies that $IJL\subseteq N$ or $KL\subseteq N.\ $
\end{theorem}

\begin{proof}
$(1)\Rightarrow(2):\ $Let $m\in(N:_{M}abc).\ $Then we have $abcm\in N.\ $If
$abcm=0,\ $then $m\in(0:_{M}abc).\ $Now, assume that $abcm\neq0.\ $As $N\ $is
a weakly classical 1-absorbing prime submodule, we have $abm\in N$ or $cm\in
N.\ $Thus, we have $m\in(N:_{M}ab)\cup(N:_{M}c).\ $This gives $(N:_{M}%
abc)=(N:_{M}ab)\cup(N:_{M}c)\cup(0:_{M}abc).\ $As $R\ $is a $um$-ring, we have
$(N:_{M}abc)=(N:_{M}ab)$ or $(N:_{M}abc)=(0:_{M}abc)$ or $(N:_{M}%
abc)=(N:_{M}c).\ $

$(2)\Rightarrow(3):\ $Let $0\neq abcL\subseteq N.$ Then we have $L\subseteq
(N:_{M}abc).\ $Since $0\neq abcL,\ (N:_{M}abc)\neq(0:_{M}abc).\ $This gives
$L\subseteq(N:_{M}abc)=(N:_{M}ab)$ or $L\subseteq(N:_{M}abc)=(N:_{M}c).\ $Then
we have $abL\subseteq N\ $or $cL\subseteq N.$

$(3)\Rightarrow(4):\ $Suppose that $abL\nsubseteq N$ and take $x\in
(N:_{R}abL).\ $Then we have $abxL\subseteq N.\ $If $abxL=0,\ $then we get
$x\in(0:_{R}abL).\ $Now, assume that $abxL\neq0.\ $Then by part (3), we
conclude that $xL\subseteq N,\ $that is, $x\in(N:_{R}L).\ $This gives that
$(N:_{R}abL)=(0:_{R}abL)\cup(N:_{R}L).\ $The follows from the fact that if a
submodule is a union of two submodules, then it must be equal to one of them.

$(4)\Rightarrow(5):\ $Let $0\neq abJL\subseteq N$\ and assume that
$abL\nsubseteq N.\ $Then we have $J\subseteq(N:_{R}abL)$ and $(N:_{R}%
abL)\neq(0:_{R}abL).\ $Thus by part (4), we get $J\subseteq(N:_{R}%
abL)=(N:_{R}L).\ $This implies that $JL\subseteq N\ $which completes the proof.

$(5)\Rightarrow(6):\ $Suppose that $0\neq aIJL\subseteq N\ $and $JL\nsubseteq
N.\ $Now, take $x\in I.\ $Since $0\neq aIJL,\ $there exists $y\in I$ such that
$ayJL\neq0.\ $As $0\neq ayJL\subseteq N,\ $by part (5), $ayL\subseteq N.\ $If
$0\neq axJL\subseteq N,\ $then similarly we have $axL\subseteq N.\ $If
$0=axJL,\ $then $0\neq ayJL=a(x+y)JL\subseteq N.\ $Then similar argument shows
that $a(x+y)L\subseteq N\ $which implies that $axL\subseteq N.\ $Hence, we
conclude that $aIL\subseteq N.\ $

$(6)\Rightarrow(7):\ $Suppose that $IJL\nsubseteq N$ and $a\in(N:_{R}%
IJL).\ $Then $aIJL\subseteq N$ and there exists $x\in J$ such that
$xIL\nsubseteq N.$\ Assume that $aIJL=0.\ $Then we have $a\in(0:_{R}%
IJL).\ $Now, assume that $aIJL\neq0.\ $Then there exists $y\in J$ such that
$0\neq ayIL=yI(Aa)L\subseteq N.\ $Then by part (6), we conclude that
$yIL\subseteq N$ or $(Aa)L\subseteq N.\ $In the second case, we have
$a\in(N:_{R}L).\ $So assume that $yIL\subseteq N.\ $If $0\neq
axIL=xI(Aa)L\subseteq N,\ $then by part (6), we have $(Aa)L\subseteq N$ which
implies that $a\in(N:_{R}L).\ $So we may assume that $axIL=0.$ In this case,
we have $0\neq ayIL=a(x+y)IL=(x+y)I(Aa)L\subseteq N.\ $Again by part (6), we
conclude that $(x+y)IL\subseteq N$ or $(Aa)L\subseteq N.\ $If
$(x+y)IL\subseteq N,\ $then we get $xIL\subseteq N$ because $yIL\subseteq
N.\ $This is a contradiction. Thus we conclude that $(Aa)L\subseteq N,\ $that
is, $a\in(N:_{R}L).\ $By above arguments, we conclude that $(N:_{R}%
IJL)\subseteq(0:_{R}IJL)\cup(N:_{R}L).\ $Since the other inclusion is always
true, we have the equality $(N:_{R}IJL)=(0:_{R}IJL)\cup(N:_{R}L)$.\ In this
case, we have $(N:_{R}IJL)=(0:_{R}IJL)$ or $(N:_{R}IJL)=(N:_{R}L).$

$(7)\Rightarrow(8):\ $Suppose that $0\neq IJKL\subseteq N$ and $IJL\nsubseteq
N$.\ Then we have $K\subseteq(N:_{R}IJL)\neq(0:_{R}IJL).$ Then by part (7), we
have $K\subseteq(N:_{R}IJL)=(N:_{R}L)$ which implies that $KL\subseteq N.\ $

$(8)\Rightarrow(1):\ $It is clear.
\end{proof}

By \cite[Theorem 4]{Zeynep}, we know that if $N\ $is a classical 1-absorbing
prime submodule of $M,\ $then $(N:_{R}M)\ $is a 1-absorbing prime ideal of
$R.\ $One can naturally ask whether this property can be extended to weakly
classical 1-absorbing prime or not. Now, we give a negative answer to this
question with the following example.

\begin{example}
Let $p$ be a prime number and consider $%
\mathbb{Z}
$-module $%
\mathbb{Z}
_{p^{2}}$. Assume that $N=(\overline{0}).\ $Then $N\ $is trivially weakly
classical 1-absorbing prime submodule. However, $(N:_{%
\mathbb{Z}
}%
\mathbb{Z}
_{p^{2}})=p^{2}%
\mathbb{Z}
$ is not a weakly 1-absorbing prime ideal.
\end{example}

The following is an immediate consequences of Theorem \ref{tmain1}.

\begin{corollary}
Let $N\ $be a proper submodule of $M.\ $Then $N\ $is a weakly classical
1-absorbing prime submodule of $M\ $if and only if for every nonunits $a,b\in
R$ and $m\in M,\ (N:_{R}abm)=(0:_{R}abm)$ or $(N:_{R}abm)=(N:_{R}m)$ or
$(N:_{R}abm)=R.$
\end{corollary}

\begin{theorem}
\label{tmult}Let $M\ $be a finitely generated multiplication module such that
$Ann_{R}(M)\ $is a 1-absorbing prime ideal of $R$. For a proper submodule $N$
of $M,\ $the followings are equivalent.

(1)\ $N\ $is a weakly classical 1-absorbing prime submodule of $M.$

(2) If $0\neq N_{1}N_{2}N_{3}N_{4}\subseteq N$ for some proper submodules
$N_{1},N_{2},N_{3}\ $of $M\ $and some submodule $N_{4}\ $of$\ M,\ $then
$N_{1}N_{2}N_{4}\subseteq N$ or $N_{3}N_{4}\subseteq N.\ $

(3)\ $(N:_{R}M)\ $is a weakly 1-absorbing prime ideal of $R.$

(4)\ $N=PM$ for some weakly 1-absorbing prime ideal $P\ $of $R$ with
$Ann_{R}(M)\subseteq P.$
\end{theorem}

\begin{proof}
$(1)\Rightarrow(2):\ $Suppose that $N\ $is a weakly classical 1-absorbing
prime submodule of $M.$ Let $0\neq N_{1}N_{2}N_{3}N_{4}\subseteq N$ for some
proper submodules $N_{1},N_{2},N_{3}\ $of $M\ $and some submodule $N_{4}%
\ $of$\ M.\ $Since $M\ $is multiplication, there exist proper ideals
$I_{1},I_{2},I_{3}$ of $R$ such that $N_{i}=I_{i}M$ for every $i=1,2,3.\ $Then
we have $0\neq I_{1}I_{2}I_{3}N_{4}\subseteq N.\ $Then by Theorem
\ref{tmain2}, $I_{1}I_{2}N_{4}=N_{1}N_{2}N_{4}\subseteq N$ or $I_{3}%
N_{4}=N_{3}N_{4}\subseteq N.\ $

$(2)\Rightarrow(1):\ $It is similar to $(1)\Rightarrow(2).$

$(1)\Rightarrow(3):\ $Suppose that $N\ $is a weakly classical 1-absorbing
prime submodule,$\ a,b,c\in R$ are nonunits and $0\neq abc\in(N:_{R}%
M).\ \ $Then we have $abcM\subseteq N.$ If $abcM=0,\ $then $ab\in Ann_{R}(M)$
or $c\in Ann_{R}(M).\ $Which implies that $abM\subseteq N$ or $cM\subseteq
N.\ $Now, assume that $0\neq abcM\subseteq N.\ $Then by Theorem \ref{tmain2},
$abM\subseteq N$ or $cM\subseteq N$ which implies that $ab\in(N:_{R}M)$ or
$c\in(N:_{R}M).$ Hence, $(N:_{R}M)\ $is a weakly classical 1-absorbing prime submodule.

$(3)\Rightarrow(4):\ $It is clear.

$(4)\Rightarrow(1):\ $Let $0\neq IJKL\subseteq N=PM$ for some proper ideals
$I,J,K$ of $R\ $and some submodule $L$ of $M.\ $Then we have $IJK(L:_{R}%
M)\subseteq(PM:_{R}M)=P+Ann_{R}(M)=P$ by \cite[Corollary to Theorem 9]{Smith}.
As $P\ $is a weakly 1-absorbing prime ideal, we have $IJ(L:_{R}M)\subseteq P$
or $K(L:_{R}M)\subseteq P.\ $Then we have $IJL\subseteq N$ or $KL\subseteq
N.\ $Then by Theorem \ref{tmain2}, $N\ $is a weakly classical 1-absorbing
prime submodule of $M.$
\end{proof}

Since the proof of the following lemma is similar to the previous proposition
$(1)\Rightarrow(3)$, we omit the proof.

\begin{lemma}
\label{lemfaith}Let $M$ be a faithful $R$-module and $N\ $a weakly classical
1-absorbing prime submodule of $M.\ $Then $(N:_{R}M)$ is a weakly 1-absorbing
prime ideal of $R.$
\end{lemma}

\begin{theorem}
\label{tcar1}Let $M_{1},M_{2}$ be two $R$-modules and $N_{1}$ be a proper
submodule of $M_{1}$. The following conditions are equivalent.

(1)$\ N=N_{1}\times M_{2}$ is a weakly classical 1-absorbing prime submodule
of $M=M_{1}\times M_{2}$.

(2)$\ N_{1}$ is a weakly classical 1-absorbing prime submodule of $M_{1}$ and
if whenever $(a,b,c,m)$ is a classical 1-quadruple-zero of $N_{1}\ $for some
nonunits $a,b,c\in R$ and $m\in M_{1},\ $then $abc\in Ann_{R}(M_{2}).$
\end{theorem}

\begin{proof}
$(1)\Longrightarrow(2):\ $Suppose that $N=N_{1}\times M_{2}$ is a weakly
classical 1-absorbing prime submodule of $M=M_{1}\times M_{2}$. Let $0\neq
abcm\in N_{1}\ $for some nonunits $a,b,c\in R$ and $m\in M_{1}.\ $Then we have
$(0,0)\neq abc(m,0)=(abcm,0)\in N.\ $Since $N\ $is a weakly classical
1-absorbing prime submodule of $M,\ $we have $ab(m,0)=(abm,0)\in N$ or
$c(m,0)=(cm,0)\in N.\ $This implies that $abm\in N_{1}\ $or $cm\in N_{1}%
.\ $Thus, $N_{1}$ is a weakly classical 1-absorbing prime submodule of
$M_{1}.\ $Now, assume that $(a,b,c,m)$ is a classical 1-quadruple-zero of
$N_{1}.\ $Then we have $abcm=0$, $abm\notin N_{1}\ $and $cm\notin N_{1}%
.\ $Now, we will show that $abc\in Ann_{R}(M_{2}).\ $Suppose that $abc\notin
Ann_{R}(M_{2}).\ $Then there exists $n\in M_{2}$ such that $abcn\neq0.\ $This
gives $(0,0)\neq abc(m,n)=(0,abcn)\in N.\ $As $N\ $is a weakly classical
1-absorbing prime, we have $ab(m,n)\in N$ or $c(m,n)\in N.\ $Which implies
that $abm\in N_{1}$ or $cm\in N_{1}.\ $Both of them are contradictions. Thus,
$abc\in Ann_{R}(M_{2}).$

$(2)\Longrightarrow(1):\ $Let $a,b,c\in R$ be nonunits and $\left(
m,n\right)  \in M=M_{1}\times M_{2}$ be such that $\left(  0,0\right)  \neq
abc\left(  m,n\right)  \in N=N_{1}\times M_{2}$. First assume that $abcm\neq
0$. Then by part (2), $abm\in N_{1}$ or $cm\in N_{1}$. So, we conclude that
$ab\left(  m,n\right)  \in N$ or $c\left(  m,n\right)  \in N.\ $This shows
that $N\ $is a weakly classical 1-absorbing prime. So assume that$\ abcm=0$.
If $abm\notin N_{1}$ and $cm\notin N_{1},\ $then $(a,b,c,m)$ is a classical
1-quadruple-zero of $N_{1}.\ $Then by part (2), $abcn=0\ $and so
$abc(m,n)=(0,0)$ which is a contradiction. Thus, we have $abm\in N_{1}$ or
$cm\in N_{1}.$ A similar argument shows that $ab(m,n)\in N$ or $c(m,n)\in
N\ $which completes the proof.
\end{proof}

\begin{proposition}
\label{pro7} Let $M_{1},M_{2}$ be two $R$-modules and $N_{1},N_{2}$ be two
proper submodules of $M_{1}$ and $M_{2},$ respectively. If $N=N_{1}\times
N_{2}$ is a weakly classical 1-absorbing prime submodule of $M=M_{1}\times
M_{2}$, then $N_{1}$ is a weakly classical 1-absorbing prime submodule of
$M_{1}$ and $N_{2}$ is a weakly classical 1-absorbing prime submodule of
$M_{2}$.
\end{proposition}

\begin{proof}
One can easily verify the claim by a similar argument in the previous theorem.
\end{proof}

The following example shows that the converse of previous proposition is not
true in general.

\begin{example}
Let $R=%
\mathbb{Z}
$, $M=%
\mathbb{Z}
^{2}$ and $N=p%
\mathbb{Z}
\times q%
\mathbb{Z}
$ where $p,q$ are two distinct prime integers. Since, $p%
\mathbb{Z}
$, $q%
\mathbb{Z}
$ are prime ideals of $%
\mathbb{Z}
$, then $p%
\mathbb{Z}
$, $q%
\mathbb{Z}
$ are weakly classical 1-absorbing prime $%
\mathbb{Z}
$-submodules of $%
\mathbb{Z}
$. Notice that $\left(  0,0\right)  \neq ppq\left(  1,1\right)  =\left(
ppq,ppq\right)  \in N$, but neither $pp\left(  1,1\right)  \in N$ nor
$q\left(  1,1\right)  \in N$. So, $N$ is not a weakly classical 1-absorbing
prime submodule of $M$.
\end{example}

Let $R_{i}$be a commutative ring with identity and $M_{i}$ be an $R_{i}%
$-module for every $i=1,2.$ Let $R=R_{1}\times R_{2}$. Then $M=M_{1}\times
M_{2}$ is an $R$-module with componentwise addition and scalar multiplication.
Also each submodule of $M$ is in the form of $N=N_{1}\times N_{2}$ for some
submodules $N_{1}$ of $M_{1}$ and $N_{2}$ of $M_{2}$.

\begin{theorem}
\label{tcar2}Let $M_{i}$ be a faithful multiplication $R_{i}$-module such that
$Ann_{R}(M_{i})$ is not unique maximal ideal of $R_{i}\ $for each $i=1,2$ $.$
Suppose that $R=R_{1}\times R_{2}\ $is a decomposable ring and $M=M_{1}\times
M_{2}.\ $Further, assume that $N=N_{1}\times N_{2}\neq0\ $for some submodule
$N_{i}\ $of $M_{i}.\ $The following statements are equivalent.

(1)\ $N\ $is a weakly classical 1-absorbing prime submodule of $M.$

(2)\ $N=N_{1}\times M_{2}\ $for some classical prime submodule $N_{1}\ $of
$M_{1}\ $or $N=M_{1}\times N_{2}\ $for some classical prime submodule
$N_{2}\ $of $M_{2}.$

(3)\ $N$ is a classical prime submodule of $M.$

(4)\ $N\ $is a weakly classical prime submodule of $M.\ $
\end{theorem}

\begin{proof}
$(1)\Rightarrow(2):\ $Suppose that $N=N_{1}\times N_{2}\ $is a weakly
classical 1-absorbing prime submodule of $M.$\ Since $M_{1},M_{2}\ $are
faithful modules, so is $M.\ $As $N\ $is a weakly classical 1-absorbing prime
submodule, by Lemma \ref{lemfaith}, $(N:_{R}M)\ $is a weakly 1-absorbing prime
ideal of $R.\ $Since $N\ $is nonzero and $M\ $is faithful multiplication
module, $(N:_{R}M)=(N_{1}:_{R_{1}}M_{1})\times(N_{2}:_{R_{2}}M_{2})$ is
nonzero. Then by the proof of \cite[Theorem 10]{Eda}, we have $(N_{1}:_{R_{1}%
}M_{1})=R_{1}$ or $(N_{2}:_{R_{2}}M_{2})=R_{2}.\ $Without loss of generality,
we may assume that $N=N_{1}\times M_{2}\ $for some proper submodule $N_{1}%
\ $of $M_{1}.\ $Now, we will show that $N_{1}\ $is a classical prime submodule
of $M_{1}.\ $Let $abm\in N_{1}$ for some $a,b\in R_{1}\ $and $m\in M_{1}%
.\ $Since $Ann_{R}(M_{2})$ is not unique maximal ideal of $R_{2},\ $there
exists a nonunit $t\notin Ann_{R}(M_{2}).\ $Then there exists $n\in M_{2}%
\ $such that $tn\neq0.\ $Then we have $0_{M}\neq
(abm,tn)=(a,1)(1,t)(b,1)(m,n)\in N.\ $As $N\ $is a weakly classical
1-absorbing prime submodule of $M,\ $we have either $(a,1)(1,t)(m,n)\in N$ or
$(b,1)(m,n)\in N$ which implies that $am\in N_{1}\ $or $bm\in N_{1}.\ $Thus,
$N_{1}\ $is a classical prime submodule of $M_{1}.\ $

$(2)\Rightarrow(3):\ $Follows from \cite[Theorem 2.35]{Mostafanasab}.

$(3)\Rightarrow(4)\Rightarrow(1):\ $Clear.
\end{proof}

\begin{theorem}
\label{theo16} Let $R$ be a um-ring and $M$ be an $R$-module.

(1)\ If $F$ is a flat $R$-module and $N$ is a weakly classical 1-absorbing
prime submodule of $M$ such that $F\otimes N\neq F\otimes M$, then $F\otimes
N$ is a weakly classical 1-absorbing prime submodule of $F\otimes M$.

(2) Suppose that $F$ is a faithfully flat $R$-module. Then $N$ is a weakly
classical 1-absorbing prime submodule of $M$ if and only if $F\otimes N$ is a
weakly classical 1-absorbing prime submodule of $F\otimes M$.
\end{theorem}

\begin{proof}
$(1):$ Let $a,b,c\in R$ be nonunits. Then we get by Theorem \ref{tmain2},
either $\left(  N:_{M}abc\right)  =\left(  0:_{M}abc\right)  $ or $\left(
N:_{M}abc\right)  =\left(  N:_{M}ab\right)  $ or $\left(  N:_{M}abc\right)
=\left(  N:_{M}c\right)  $. Assume that $\left(  N:_{M}abc\right)  =\left(
0:_{M}abc\right)  $ . Then by \cite[Lemma 3.2]{Azizi}, we have $\left(
F\otimes N:_{F\otimes M}abc\right)  =F\otimes\left(  N:_{M}abc\right)
=F\otimes\left(  0:_{M}abc\right)  =\left(  F\otimes0:_{F\otimes M}abc\right)
=\left(  0:_{F\otimes M}abc\right)  $. Now, suppose that $\left(
N:_{M}abc\right)  =\left(  N:_{M}ab\right)  $. Again by \cite[Lemma
3.2]{Azizi}, we have $\left(  F\otimes N:_{F\otimes M}abc\right)
=F\otimes\left(  N:_{M}abc\right)  =F\otimes\left(  N:_{M}ab\right)  =\left(
F\otimes N:_{F\otimes M}ab\right)  $. Similarly, we can show that if
\ $\left(  N:_{M}abc\right)  =\left(  N:_{M}c\right)  $, then $\left(
F\otimes N:_{F\otimes M}abc\right)  =$ $\left(  F\otimes N:_{F\otimes
M}c\right)  $. Consequently, by Theorem \ref{tmain2}, we deduce that $F\otimes
N$ is a weakly classical 1-absorbing prime submodule of $F\otimes M$.

$(2):$ Let $N$ be a weakly classical 1-absorbing prime submodule of $M$\ and
assume $F\otimes N=F\otimes M$. Then $0\rightarrow F\otimes N\overset
{\subseteq}{\rightarrow}F\otimes M\rightarrow0$ is an exact sequence. Since
$F$ is a faithfully flat $R$-module, $0\rightarrow N\rightarrow M\rightarrow0$
is an exact sequence. So $N=M$ , which is a contradiction. Thus we have
$F\otimes N\neq F\otimes M$. Then $F\otimes N$ is a weakly classical
1-absorbing prime submodule of $F\otimes M$ by (1). Now for the converse, let
$F\otimes N$ be a weakly classical 1-absorbing prime submodule of $F\otimes
M$. Then we have $F\otimes N\neq F\otimes M$ and so $N\neq M$ . Let $a,b,c\in
R$ be nonunits. Then $\left(  F\otimes N:_{F\otimes M}abc\right)  =\left(
0:_{F\otimes M}abc\right)  $ or $\left(  F\otimes N:_{F\otimes M}abc\right)
=\left(  F\otimes N:_{F\otimes M}ab\right)  $ or $\left(  F\otimes
N:_{F\otimes M}abc\right)  =\left(  F\otimes N:_{F\otimes M}c\right)  $ by
Theorem \ref{tmain2}. Assume that $\left(  F\otimes N:_{F\otimes M}abc\right)
=\left(  0:_{F\otimes M}abc\right)  $. Hence $F\otimes\left(  N:_{M}%
abc\right)  =\left(  F\otimes N:_{F\otimes M}abc\right)  =\left(  0:_{F\otimes
M}abc\right)  =F\otimes\left(  0:_{M}abc\right)  $. So $0\rightarrow
F\otimes\left(  0:_{M}abc\right)  \overset{\subseteq}{\rightarrow}%
F\otimes\left(  N:_{M}abc\right)  \rightarrow0$ is exact sequence. Since $F$
is a faithfully flat $R$-module $0\rightarrow\left(  0:_{M}abc\right)
\overset{\subseteq}{\rightarrow}\left(  N:_{M}abc\right)  \rightarrow0$ is an
exact sequence which implies that $\left(  N:_{M}abc\right)  =\left(
0:_{M}abc\right)  $. By a similar argument, we can deduce that if $\left(
F\otimes N:_{F\otimes M}abc\right)  =\left(  F\otimes N:_{F\otimes
M}ab\right)  $ or $\left(  F\otimes N:_{F\otimes M}abc\right)  =\left(
F\otimes N:_{F\otimes M}c\right)  ,$ then $\left(  N:_{M}abc\right)  =\left(
N:_{M}ab\right)  $ or $\left(  N:_{M}abc\right)  =\left(  N:_{M}c\right)  $.
Consequently, $N$ is a weakly classical 1-absorbing prime submodule of $M$ by
Theorem \ref{tmain2}.
\end{proof}

\begin{corollary}
Let $R$ be an um-ring, $M$ \ be an $R$-module, and $x$ be an indeterminate. If
$N$ is a weakly classical 1-absorbing prime submodule of $M$, then $N\left[
x\right]  $ is a weakly classical 1-absorbing prime submodule of $M\left[
x\right]  $.
\end{corollary}

\begin{proof}
Assume that $N$ is a weakly classical 1-absorbing prime submodule of $M$.
Notice that $R\left[  x\right]  $ is a flat $R$-module. So by Theorem
\ref{theo16}, $N\left[  x\right]  \simeq R\left[  x\right]  \otimes N$ is a
weakly classical 1-absorbing prime submodule of $M[x]\simeq$ $R\left[
x\right]  \otimes M$.
\end{proof}

Now, we investigate the modules over which every proper ideal is weakly
classical 1-absorbing prime.

\begin{theorem}
\label{tfinal}Let $M\ $be an $R$-module. The following statements are satisfied.

(1) If every proper submodule of $M\ $is a weakly classical 1-absorbing prime.
Then $Jac(R)^{3}M=0.$

(2)\ Suppose that $(R,\mathfrak{m})$ is a local ring. Then every proper
submodule of $M\ $is a weakly classical 1-absorbing prime if and only if
$\mathfrak{m}^{3}M=0.\ $
\end{theorem}

\begin{proof}
$(1):\ $Suppose that every proper submodule of $M\ $is a weakly classical
1-absorbing prime. Now, we will show that $Jac(R)^{3}M=0.$ Suppose to the
contrary. Then there exist $x,y,z\in Jac(R)$ and $m\in M$ such that
$xyzm\neq0.\ $Now, put $N=Rxyzm.\ $Since $0\neq xyzm\in N\ $and $N\ $is a
weakly classical 1-absorbing prime submodule, we have either $xym\in N$ or
$zm\in N.\ $Then there exists $a\in R$ such that $xym=axyzm$ or $zm=axyzm.\ $%
In the first case, we have $(1-az)xym=0.\ $Since $1-az\ $is unit, we conclude
that $xym=0.\ $In the second case, one can similarly show that $zm=0.\ $Thus,
we conclude, $xyzm=0\ $which is a contradiction. Hence, $Jac(R)^{3}M=0.$

$(2):\ $Suppose that $(R,\mathfrak{m})$ is a local ring. If every proper
submodule of $M\ $is a weakly classical 1-absorbing prime, then by (1),
$\mathfrak{m}^{3}M=0.\ $For the converse, note that $xyz=0\ $for every
nonunits $x,y,z\in R.\ $In this case, every proper submodule of $M\ $is
trivially weakly classical 1-absorbing prime submodule.
\end{proof}

\end{document}